\documentclass[12pt]{amsart}

\newtheorem{theorem}{Theorem}[section]
\newtheorem{proposition}[theorem]{Proposition}
\newtheorem{lemma}[theorem]{Lemma}

\def\C{{\mbox{\rm\kern.24em
\vrule width.03em height1.43ex depth-.052ex \kern-.26em C}}}
\def\QSet{\mbox{\rm\kern.24em
\vrule width.03em height1.48ex depth-.051ex \kern-.26em Q}}
\def\Z{{\mbox{\rm\kern.25em
\vrule width.03em height0.57ex depth0ex
\kern.033em
\vrule width.03em height1.52ex depth-0.96ex \kern-.338em Z}}}
\def\N{{{\mbox{\rm I\kern-.2em N}}_0}}
\def\M{{{\mbox{\rm I\kern-.2em M}}}}
\def\PSet{\mbox{\rm I\kern-.22em P}}
\def\R{{\mbox{\rm I\kern-.22em R}}}

\def\I{{\bf I}}

\def\I{{\bf I}}

\def\<{\left<}
\def\>{\right>}

\newcommand {\BBC}{\mathbb{C}}
\newcommand {\diam}{\mathop{\mathrm{diameter}}}
\newcommand {\rea}{\mathbb{R}}

\newcommand {\dis}{\mathop{\mathrm{dist}}}

\newcommand {\calK}{\mathcal{K}}
\newcommand {\calB}{\mathcal{B}}
\newcommand {\calV}{\mathcal{V}}
\newcommand {\calM}{\mathcal{M}}

\title{A Calder\'on Zygmund decomposition for multiple frequencies and
an application to an extension of a Lemma of Bourgain.}

\author
[F. Nazarov, R. Oberlin, C. Thiele]
{Fedor Nazarov \ \ \ Richard Oberlin \ \ \ Christoph Thiele }

\address{F. Nazarov, Department of Mathematics,
University of Wisconsin, Madison, 
480 Lincoln Drive, Madison, WI 53706, USA}
\email{nazarov@math.wisc.edu}

\address{R. Oberlin, Department of Mathematics,
UCLA, Los Angeles, CA 90095-1555, USA}
\email{oberlin@math.ucla.edu}

\address{C. Thiele, Department of Mathematics,
UCLA, Los Angeles, CA 90095-1555, USA}
\email{thiele@math.ucla.edu}

\thanks{
F.N. partially supported by NSF grant DMS 0800243.
R.O. partially supported by NSF VIGRE grant
DMS 0502315.
C.Th. partially supported by NSF grant DMS 0701302.}

\date{\today}

\begin{document}

\maketitle

\begin{abstract}
We introduce a Calder\'on Zygmund decomposition such that the bad function
has vanishing integral against a number of pure frequencies.  Then we prove
a variation norm variant of a maximal inequality for several frequencies
due to Bourgain. To obtain the full range of $L^p$ estimates we apply the
multi frequency Calder\'on Zygmund decomposition.
\end{abstract}

\section{Introduction}

The Calder\'on Zygmund decomposition is a technique to extend
bounds for operators $T$ acting on some $L^p$ space to bounds of $T$ acting on 
$L^q$ spaces with lower exponent $1<q<p$. In the most basic example one decomposes 
a function $f\in L^1(\R)$ as the sum of a good function $g$ and a bad function $b$, where 
the good function is in  $L^p(\R)$ and we can apply 
the known bounds, while the bad function $b$ is the sum of localized functions 
$b=\sum_{I\in \I} b_I$ parameterized by a collection $\I$ of disjoint intervals of controlled length 
such that each $b_I$ is supported on the interval $I$ and satisfies the cancellation condition 
\begin{equation}\label{cancellation}
\int b_I(x)\, dx =0\ \ \ .
\end{equation}
The crucial point is that one can use the cancellation condition (\ref{cancellation}) to obtain good
estimates for $T(b_I)$ away from the interval $I$.

In this paper we propose a variant of the Calder\'on Zygmund decomposition, where the mean zero
condition is replaced by a collection of conditions (\ref{meanzero})
for a number of frequencies $\xi_1,\dots,\xi_N$.  Estimates on the good and bad function 
depend on the number $N$ of frequencies, and good control on the 
$N$- dependence is the essence of the matter.

\begin{theorem}\label{cztheorem}
There is a universal constant $C$ such that the following holds.
Let $\xi_1<\dots<\xi_N$ be arbitrary real numbers for some $N\ge 1$. Let $f\in L^1(\R)$
and let $\lambda>0$. Then there is a decomposition 
$$f=g+\sum_{I\in \I} b_I$$ 
for some disjoint collection $\I$ of intervals with
$$\sum_{I\in \I} |I|\le C N^{1/2} \|f\|_1\lambda^{-1}$$ 
such that for each $I\in \I$ and $1\le j\le N$ we have the following, where $f_I$
is the product of the function $f$ with the characteristic function of the interval $I$:
\begin{equation}\label{gbound}
\|g\|_2^2\le C \|f\|_1 N^{1/2} \lambda 
\end{equation}
\begin{equation}\label{fibound}
\|f_I\|_1\le C|I|\lambda
\end{equation}
\begin{equation}\label{fibibound}
\|f_I - b_I\|_{2} \le C|I|^{1/2} \lambda N^{1/2}
\end{equation}
\begin{equation}\label{meanzero}
\int b_I(x) e^{i\xi_j x}\, dx=0
\end{equation}
and the support of $b_I$ is $3I$, the interval with the same center as $I$ and three times
the length.  
\end{theorem}

The exponents of $N$ in this theorem are optimal.  The condition (\ref{meanzero})
means that the functions $f_I$ and $f_I-b_I$ induce the same linear functional on the subspace 
$H$ of $L^2(I)$ spanned  by the functions $e^{i\xi_jx}$. The Riesz representation theorem then
provides the optimal choice of $f_I-b_I$ as an element in $H$. The desired bounds
for $f_I-b_I$ follow from an elegant estimate by Borwein and Erdelyi \cite{borwein06nti}. 
A different approach to proving these bounds is to find a perturbation of the inner product of 
$H$ which permits an orthonormal basis consisting of functions with universally bounded 
$L^\infty$ norm, independent of $N$. We are able to construct such a basis in the well separated 
case that is thoroughly discussed in \cite{demeter09osm}, namely
$\xi_j-\xi_{j-1}>|I|^{-1}$ for all $j$, 
and in the well localized case when $\xi_N-\xi_1\le C_\epsilon N^{1-\epsilon}|I|^{-1}$.
However, we do not know a construction for such a basis in general, and the strength of the
argument by Borwein and Erdelyi is to circumvent the need for it.

We anticipate this Calder\'on Zygmund decomposition or variants thereof to be applicable in 
an array of problems in time-frequency analysis, where one often needs integral conditions 
such as (\ref{meanzero}) for several frequencies. For example, in \cite{oberlin09w}, two of the 
authors use a simple and explicit discrete variant of this Calder\'on Zygmund  decomposition 
to prove hitherto unknown uniform bounds for a discrete model of the bilinear Hilbert transform. 

In this paper we use Theorem \ref{cztheorem} to prove an extension of a 
multi-frequency maximal inequality of Bourgain
(\cite{bourgain89pet}) that has played a role in time-frequency analysis and in proving pointwise 
convergence results for various ergodic averages. For each dyadic interval 
$$\omega = [2^kn,2^k(n+1))$$
with $k,n\in \Z$ let $\phi_{\omega}$ be a Schwartz function whose Fourier transform $\hat{\phi}_{\omega}$ is supported on $\omega.$  Let $\xi_1<\dots<\xi_N$ be real numbers and denote by $X$ the set $\{\xi_1,\dots,\xi_N\}$.
We are interested in bounds for the vector valued operator 
\[
\Delta_k[f] = \sum_{\substack{|\omega| = 2^k \\ \omega \cap X \neq \emptyset}} f * \phi_{\omega}
\] 
whose vector components are parameterized by the integer $k$.
For an exponent $1<r<\infty$, define the $r$-variation semi-norms of a 
sequence $g_k$ by
\begin{equation} \label{rtildevariationdef}
\|g_k\|_{\tilde{V}^r_k} := \sup_{M, k_0 < \ldots < k_M} \left(\sum_{j = 1}^M |g(k_j) - g(k_{j-1})| ^r\right)^{1/r}
\end{equation}
where the supremum is over all strictly increasing finite sequences $k_j$ of arbitrary finite length $M+1$ and define the variation norm 
\[
\|g_k\|_{V^r_{k}} := \sup_{k} |g_k| + \|g_k\|_{\tilde{V}^r_{k}}\ \ .
\]
When $r=\infty,$ we replace the sum (\ref{rtildevariationdef}) by a supremum  in the usual manner.

It was proven in \cite{demeter08bdr} that for $r>0$ we have
\begin{equation} \label{mml2bound}
\|\Delta_k[f](x)\|_{L^2_x(L^\infty_k)} \leq (1 + \log(N)) N^{\frac{1}{2} - \frac{1}{r}}(D_1 + \sup_{j = 1, \ldots, N}\| \sum_{|\omega| = 2^k}\hat{\phi}_{\omega}(\xi_j)\|_{V^r_k}) \|f\|_{L^2}
\end{equation}
with the convention
\[
D_M : = \sup_{\omega,x} |\omega|^{M}|\hat{\phi}_{\omega}^{(M)}(x)| 
\]
for any integer $M \geq 0$ where the supremum is over all dyadic intervals $\omega$, real numbers $x$, 
and where $\hat{\phi}_\omega^{(M)}$ is the $M$'th derivative of $\hat{\phi}_{\omega}$.

This is a weighted version of the above mentioned bound of Bourgain's originating in \cite{bourgain89pet}. Our aim is to strengthen (\ref{mml2bound}) in two directions. First, we would like to replace $L^2$ by $L^p$ for $1 \leq p < 2$; this is the final step of the proof of the $L^p$ return times theorem initiated in \cite{demeter09osm}, \cite{demeter09irt}. Note that \cite{demeter09osm} proves such an extension in the case of separated frequencies; it also suggests a line of reasoning for the general case, however we have been unable to complete the general case without the use of the multi-frequency Calder\'on Zygmund condition. Second we would like to replace the $L^\infty_k$ norm by the stronger $q$-variation norm. Specifically, we will show:

\begin{theorem} \label{vmt}
Suppose $1 < p \leq 2 < r < q$. Then, there exists an $M$ depending only on $q$ and $r$ such that 
\begin{multline*}
\|\Delta_k[f](x)\|_{L^p_x(V^q_k)} \leq \\ C_{p,q,r} (1 + \log(N))N^{(\frac{1}{2} - \frac{1}{r})\frac{q}{q-2} + \frac{1}{p} - \frac{1}{2}}(D_{M} + \sup_{j = 1, \ldots, N}\| \sum_{|\omega| = 2^k}\hat{\phi}_{\omega}(\xi_j)\|_{V^r_k}) \|f\|_{L^p}.
\end{multline*} 
\end{theorem}

In applications, for each $q$ one takes $r$ near $2$ to obtain exponents arbitrarily close to $N^{\frac{1}{p} - \frac{1}{2}}.$ We expect the full strength of the variational estimate to be used in forthcoming
work by the second author.

We will start with the short proof of Theorem \ref{cztheorem} in Section \ref{czproof}.
We then prove Theorem \ref{vmt} for the exponent $p=2$ in Sections \ref{trigsumsection} and 
$\ref{rtvvfsection}$; the main ingredient necessary here to improve (\ref{mml2bound}) to a variational bound is an estimate for exponential sums proven in Section \ref{trigsumsection}. In Section \ref{czsection}, we extend the bound to cover exponents $1 < p < 2$ by proving a weak-type estimate at $p=1;$ the main ingredient 
here is the Calder\'on Zygmund decomposition of Theorem \ref{cztheorem}.

\section{A multiple frequency Calder\'{o}n Zygmund decomposition} \label{czproof}

\begin{proof}[Proof of Theorem \ref{cztheorem}.]

Let $f\in L^1(\R)$. Consider the set
$$E=\{x:\calM[f](x)>\lambda N^{-1/2}\}$$
where $\calM$ is the Hardy Littlewood maximal operator. By the Hardy Littlewood maximal theorem we have
$$|E|\le C N^{1/2} \|f\|_{L^1} \lambda^{-1}\ \ .$$
Let $\I$ be the collection of maximal dyadic intervals contained in $E$ such
that $6I$ is also contained in $E$ (here and in the rest of the proof, $CI$ denotes the dilate of $I$ with respect to the center of $I$). Clearly the collection $\I$ covers $E$,
and the collection of intervals $3I$ has bounded overlap.
 
Consider the finite dimensional subspace 
\begin{equation}\label{span}{\rm span}\{e^{i\xi_j x}: 1\le j\le N\}
\end{equation}
of the Hilbert space $L^2(3I)$. For each element $v$ in this space, 
Borwein and Erdelyi prove in \cite{borwein06nti} the estimate
\begin{equation}\label{beestimate}
\|v\|_{L^\infty(I)}\le N^{1/2} |I|^{-1/2} \|v\|_{L^2(3I)}\ \ .
\end{equation}
For the convenience of the reader, we sketch the elegant proof in \cite{borwein06nti}.
Let $v_1,\dots,v_N$ be an orthonormal basis of the space (\ref{span})
considered as subspace of $L^2(I)$.
Since
$$\int_{I} \sum_{j=1}^N |v_j(x)|^2\, dx = N$$
there exists a point $x_0\in I$ such that 
$$|I| \sum_{j=1}^N |v_j(x_0)|^2\le N \ \ .$$
Hence, for every element in (\ref{span}),
$$|v(x_0)|\le \sum_{j=1}^N |\<v,v_j\> v_j(x_0)|$$
$$\le (\sum_{j=1}^N |\<v,v_j\>|^2)^{1/2} (\sum_{j=1}^N |v_j(x_0)|^2)^{1/2}\le N^{1/2} |I|^{-1/2} \|v\|_{L^2(I)}\ \ .$$
To estimate $v$ at a general point $x_1\in I$, we apply this estimate
to 
$$\tilde{v}(x)=v(x+x_0-x_1)$$ 
which is also in the space (\ref{span}) 
and thus obtain (\ref{beestimate}).

Estimate (\ref{beestimate}) implies that the function $f_I$ defines 
linear functional on the subspace (\ref{span}) of $L^2(3I)$ with norm 
bounded by $\|f_I\|_1$.
By the Riesz representation theorem, there is an element $g_I$ in this subspace such that
\[
\int f_I(y) e^{2\pi i \xi_j y}\ dy = \int_{3I} g_I(y) e^{2\pi i \xi_j y}\ dy
\]
and such that 
\[
\|g_I\|_{L^2(3I)} \leq N^{1/2} |I|^{-1/2} \|f\|_{L^1(I)}
\]
We extend $g_I$ to a function on all of $\R$ by setting it equal to $0$ outside $3I$.

For each $I\in \I$, consider the restriction $f_I$ of $f$ to $I$ and observe
that by looking at the maximal function on $12I$  we have
$$\|f_I\|_{L^1}\le 24|I| \lambda N^{-1/2}.$$

Define 
\begin{align*}
b_I &= f_I-g_I \\
b &=\sum_{\I} b_I \\
g &= f- b
\end{align*}
Observe that $b$ is supported on the set $E$.

The functions $g_I$ have bounded overlap, hence
\begin{align*}
\|g\|_{L^2}^2 &\le \int_{E^c} |f(x)|^2\, dx+ \int (\sum_I g_I(x))^2\, dx \\
&\le \int_{E^c} |f(x)| \lambda N^{-1/2} \, dx+ C \sum_I \int g_I(x)^2\, dx \\
&\le \|f\|_{L^1} \lambda N^{-1/2} + C \sum_I |I|\lambda^2 \\
&\le \|f\|_{L^1} \lambda N^{-1/2} + C |E| \lambda^2 \\
&\le C  \|f\|_{L^1} N^{1/2} \lambda .
\end{align*}

\end{proof}

\section{A variational estimate for exponential sums} \label{trigsumsection}
We recall the following lemma which was proven in \cite{demeter08bdr} and was inspired by an argument of Bourgain \cite{bourgain89pet} (See also Proposition 4.2 of \cite{demeter08wmc} which is similar to the lemma in \cite{demeter08bdr}, but given in a purely functional-analytic setting). 

\begin{lemma} \label{mets} Suppose that $\xi_1 < \ldots < \xi_N$ are real numbers, $\{c_k\}_{k=1}^{\infty}$ is a sequence in $\rea^N$, and $r > 2.$ Then
\begin{equation} 
\| \|\sum_{j = 1}^N c_{k,j} e^{2\pi i \xi_j y}\|_{L^{\infty}_k}\|_{L^2_y([0,1])} \leq C  N^{\frac{1}{2} - \frac{1}{r}} \|c_k\|_{V_k^r(l^2(\rea^N))}
\end{equation}
where $C$ may depend on $r$ and $\min_j |\xi_j - \xi_{j-1}|.$
\end{lemma}

Here we have used the obvious extension of the definition of the $r$-variation norm 
to a  function $g$ defined on a subset $\calK$ of $\rea$ which takes values in a Banach space 
$\calB$ as
\begin{equation} \label{rtildebanach}
\|g\|_{\tilde{V}^r_{k \in \calK}(\mathcal{B})} = \sup_{M, k_0 < \ldots < k_M} \left(\sum_{j = 1}^M \|g(k_j) - g(k_{j-1})\|_{\calB}^r\right)^{1/r}
\end{equation}
where the supremum is over all strictly increasing finite sequences in $\calK$, and
\[
\|g\|_{V^r_{k \in \calK}(\calB)} = \sup_{k \in \calK} \|g(k)\|_{\calB} + \|g\|_{\tilde{V}^r_{k \in \calK}(\calB)}.
\]
When $r=\infty,$ we replace the sum (\ref{rtildebanach}) by a supremum  in the usual manner.
When $\calK$ and $\calB$ are supressed, one may usually assume that they are the domain of $g$ and $\BBC$ respectively.

The crucial step towards obtaining bounds for the $V^q$ norm in Theorem \ref{vmt} is to see that Lemma \ref{mets} holds with the $L^\infty$ norm replaced by a $V^q$ norm. We thus want to prove 

\begin{lemma} \label{vets} Suppose that $\xi_1 < \ldots < \xi_N$ are real numbers, $\{c_k\}_{k=1}^{\infty}$ is a sequence in $\rea^N$, and $2 < r < q.$ Then
\begin{equation} \label{vtsbound}
\| \|\sum_{j = 1}^N c_{k,j} e^{2\pi i \xi_j y}\|_{V^q_k(\BBC)}\|_{L^2_y([0,1])} \leq C  N^{\left(\frac{1}{2} - \frac{1}{r}\right)\frac{q}{q-2}}\|c_k\|_{V_k^r(l^2(\rea^N))}
\end{equation}
where $C$ may depend on $r,q$ and $\min_j |\xi_j - \xi_{j-1}|.$
\end{lemma}

We will require use of the estimate
\begin{equation} \label{orthsums}
\|\sum_{j = 1}^N d_j e^{2\pi i \xi_j y}\|_{L^2_y([0,1])} \leq C \|d_j\|_{l^2_j}
\end{equation}
where the constant depends on $\min_j |\xi_j - \xi_{j-1}|$. 
To see this, estimate the $L^2$ norm on the left hand side by 
the norm $L^2(w)$ for some appropriate smooth weight supported on a larger
interval than $[0,1]$ and use almost orthogonality of the exponential 
functions in the space $L^2(w)$.

\begin{proof}[Proof of Lemma \ref{vets}]
By Lemma \ref{mets}, it suffices to prove (\ref{vtsbound}) with the $\tilde{V}^q$ norm in place of the $V^q$ norm. By a limiting argument, we may also assume that our sequence $\{c_k\}_{k=1}^M$ has finite length, provided that $C$ is independent of $M$. 

For each $\lambda > 0$ we cover $\{c_k\}_{k=1}^M$ with respect to $\lambda$-jumps as follows. Set $l(\lambda,1) = 1.$ Suppose that $l(\lambda,1) < \ldots <  l(\lambda,L)$ have been chosen, and let $B(c_{l(\lambda,L)},\lambda)$ denote the ball of radius $\lambda$ centered at $c_{l(\lambda,L)}.$ If $\{c_k : k > l(\lambda,L)\} \subset B(c_{l(\lambda,L)},\lambda)$ then stop and set $L_\lambda = L$ and $l(\lambda,L+1) = \infty.$ Otherwise, let $l(\lambda,L+1)$ be chosen minimally with $l(\lambda,L+1) > l(\lambda,L)$ and $c_{l(\lambda,L+1)} \notin B(c_{l(\lambda,L)},\lambda).$ This process will stop, yielding some $L_{\lambda} \leq M$. It is clear that 
\begin{equation} \label{mefromvar}
\lambda (L_{\lambda} - 1)^{1/r} \leq \|c_k\|_{V^r_k}\ \ .
\end{equation}

We now define a recursive ``parent'' function based on the covering above. 
Fix some $\lambda_0 < \min\{\|c - c'\|_{l^2(\rea^N)} : c,c' \in \{c_k\}_{k=1}^M \text{\ and } c \neq c' \}.$
For $k=1, \ldots, M$ define $\rho(-1,k) = k$. Once $\rho(n,k)$ has been defined for $n=-1, \ldots, L$ set $\rho(L+1,k) = l(2^{L+1}\lambda_0,m)$ where $m$ is the unique integer satisfying
\[
l(2^{L+1}\lambda_0,m) \leq \rho(L,k) < l(2^{L+1}\lambda_0,m+1).
\]
Notice that we have
\[
\|c_{\rho(n,k)} - c_{\rho(n+1,k)}\|_{l^2(\rea^N)} < 2^{n+1}\lambda_0
\]
and in particular $c_{\rho(0,k)} = c_k.$ Also note that $\rho(n,k) = 1$  whenever $2^n\lambda_0 \geq \diam(\{c_k\}_{k=1}^M).$ Thus
\[
c_k = c_1 + \sum_{n=0}^{\infty} c_{\rho(n,k)} - c_{\rho(n+1,k)}\ \ .
\]
Finally, by induction, one sees that $\rho(n,k)$ is nondecreasing in $k$ for each fixed $n$.

We have
\[
\|\|\sum_{j = 1}^N c_{k,j} e^{2\pi i \xi_j y}\|_{\tilde{V}^q_k(\BBC)}\|_{L^2_y([0,1])}\]
\[ \leq \sum_{n = 0}^\infty \|\|\sum_{j = 1}^N (c_{\rho(n,k),j} - c_{\rho(n+1,k),j}) e^{2\pi i \xi_j y}\|_{\tilde{V}^q_k(\BBC)}\|_{L^2_y([0,1])}\ \ .
\]
Observe that the right hand side above
\[
= \sum_{n : L_{2^n\lambda_0} > 1} \|\|\sum_{j = 1}^N (c_{\rho(n,k),j} - c_{\rho(n+1,k),j}) e^{2\pi i \xi_j y}\|_{\tilde{V}^q_k(\BBC)}\|_{L^2_y([0,1])}.
\]
Using the monotonicity of the $\rho(n,\cdot)$ 
 and the fact that the range of $\rho(n,\cdot)$ is contained in $\{l(2^n\lambda_0,m) : m = 1, \ldots, L_{2^n\lambda_0}\}$ we see that the display above is 
\[
\leq 2 \sum_{n: L_{2^n\lambda_0} > 1}\| \left(\sum_{m = 1}^{L_{2^n\lambda_0}} |\sum_{j=1}^N(c_{l(2^n\lambda_0,m),j} - c_{\tilde{\rho}(l(2^n\lambda_0,m)),j}) e^{2\pi i \xi_j y}|^q \right)^{1/q} \|_{L^2_y([0,1])}
\]
where we let $\tilde{\rho}(l(2^n\lambda_0,m))$ denote $l(2^{n+1}\lambda_0,i)$ where $i$ is the unique integer satisfying
\[
l(2^{n+1}\lambda_0,i) \leq l(2^n\lambda_0,m) < l(2^{n+1}\lambda_0,i + 1).
\]
Estimating $l^q$ by $l^2$, switching the order of integration, and using (\ref{orthsums}), we see that the $n$'th term in the outer sum above is
\[
\leq C 2^{n}\lambda_0 L_{2^n\lambda_0}^{1/2} 
\leq C (2^n\lambda_0)^{1 - \frac{r}{2}} \|c_k\|_{V^r_k(l^2(\rea^N))}^{\frac{r}{2}}.
\]

We can also estimate the $n$'th term by
\begin{multline*}
\| \left(\sum_{m = 1}^{L_{2^n\lambda_0}} (\sum_{j=1}^N|c_{l(2^n\lambda_0,m),j} - c_{\tilde{\rho}(l(2^n\lambda_0,m)),j}|)^q \right)^{1/q} \|_{L^2_y([0,1])} \\\leq N^{1/2} \| \left(\sum_{m = 1}^{L_{2^n\lambda_0}} \|c_{l(2^n\lambda_0,m)} - c_{\tilde{\rho}(l(2^n\lambda_0,m))}\|_{l^2(\rea^N)}^q \right)^{1/q} \|_{L^2_y([0,1])}
\\ \leq N^{1/2} (2^n\lambda_0)^{1 - \frac{r}{q}} \|c_k\|_{V^r_k(l^2(\rea^N))}^{\frac{r}{q}}.
\end{multline*}
Choosing whichever of the two bounds is favorable for each $n$ and summing gives the desired result.
\end{proof}

\section{The $L^2$ bound} \label{rtvvfsection}

Following the method of \cite{demeter08bdr}, our proof of Theorem \ref{vmt} for the exponent $p=2$ has two steps. We first demonstrate the bound under a certain assumption of frequency separation, and then we use a Rademacher-Menshov type argument to leverage the frequency-separated bound to give the general result.

\subsection{Frequency separated case}

We want to show the following

\begin{proposition} \label{frequencyseparatedtheorem}
Suppose that for each $j,$ $\xi_{j+1} > \xi_j + 1$ and that $2 < r < q$. Then
\begin{equation} \label{fsbound}
\|\Delta_k[f](x)\|_{L^2_x(V^q_{k \leq 0})} \leq C_{q,r} N^{(\frac{1}{2} - \frac{1}{r})\frac{q}{q-2}}(D_1 + \sup_{j}\|\sum_{|\omega| = 2^k} \hat{\phi}_{\omega}(\xi_j)\|_{V^r_{k \leq 0}}) \|f\|_{L^2}\ \ . 
\end{equation}
\end{proposition}

\begin{proof}First, we will use an averaging argument combined with Lemma \ref{vets} to reduce matters to the case $N=1.$ We then treat the single frequency case using L\'{e}pingle's inequality.

For the remainder of the proof, all $V^q,V^r,$ and $l^2$ norms will be restricted to the indices $k \leq 0.$ After renormalizing, we may assume that $D_1 = 1.$ For each $j,k$ let $\omega_{j,k}$ be the dyadic interval of length $2^k$ containing $\xi_j$ and let $\phi_{j,k} = \phi_{\omega_{j,k}}.$ Since each relevant $k \leq 0$ and each $\xi_{j+1} > \xi_{j} + 1$, we have
\[
\Delta_k[f](x) = \sum_{j = 1}^N \phi_{j,k}*f(x).
\]
Writing $\tilde{\phi}_{j,k}(x) = e^{-2\pi i\xi_j x}\phi_{j,k}(x)$ and $\tilde{f}_j(x) = e^{-2\pi i\xi_j x} (1_{\omega_{j,1}}\hat{f})\check{\ }(x)$ one sees that the right hand side above is equal to
\[
 \sum_{j = 1}^N e^{2\pi i \xi_{j}x} \tilde{\phi}_{j,k}*\tilde{f}_j(x).
\]

Let $B$ be the smallest constant for which the bound
\[
\|\sum_{j = 1}^N e^{2\pi i \xi_{j}x} \tilde{\phi}_{j,k}*g_j(x)\|_{L^2_x(V^q_k)} \leq B \|g_j(x)\|_{L^2_x(l^2_j)}.
\]
holds for every $g_j(x) \in L^2_x(l^2_j)$ such that $\hat{g}_j$ is supported on $[-1,1]$ for every $j$.
Since each $\hat{g}_j$ is supported on $[-1,1]$, we have
\[
\|g_j(x) - g_j(x - y) \|_{L^2_x(l^2_j)} \leq C |y| \|g_j(x)\|_{L^2_x(l^2_j)}.
\]
Averaging over small values of $y$
\begin{multline*}
\|\sum_{j=1}^N e^{2 \pi i \xi_j x} \tilde{\phi}_{j,k} * g_j(x)\|_{L^2_x(V^q_k)} \leq \\
C \|\sum_{j=1}^N  e^{2 \pi i \xi_j x} \tilde{\phi}_{j,k} * (g_j(\cdot - y))(x)\|_{L^2_x(L^2_{y \in [0,\epsilon]}(V^q_k))}
+ \frac{B}{2} \|g_j(x)\|_{L^2_x(l^2_j)}.
\end{multline*} 
Making the right hand side larger by replacing $L^2([0,\epsilon])$
by $L^2([0,1])$ and using translation invariance, the right hand side
can be estimated by
\[
C \|\sum_{j=1}^N  e^{2 \pi i \xi_j y} e^{2 \pi i \xi_j x} \tilde{\phi}_{j,k} * g_j(x)\|_{L^2_x(L^2_{y \in [0,1]}(V^q_k))}
+ \frac{B}{2} \|g_j(x)\|_{L^2_x(l^2_j)}.
\]

Applying Lemma \ref{vets}, we have
\begin{align*}
& \|\sum_{j=1}^N e^{2 \pi i \xi_j y} e^{2 \pi i \xi_j x} \tilde{\phi}_{j,k} * g_j(x)\|_{L^2_x(L^2_{y \in [0,1]}(V^q_k))} \\
\leq \ &  C N^{\left(\frac{1}{2} - \frac{1}{r}\right)\frac{q}{q-2}} \|e^{2 \pi i \xi_j x} \tilde{\phi}_{j,k} * g_j(x)\|_{L^2_x(V^r_k(l^2_j))} \\
\leq \ &C N^{\left(\frac{1}{2} - \frac{1}{r}\right)\frac{q}{q-2}} \|\tilde{\phi}_{j,k} * g_j(x)\|_{l^2_j(L^2_x(V^r_k))}\ \ .
\end{align*}
Below, we will show that for each $j$
\begin{equation} \label{caseN1}
\|\tilde{\phi}_{j,k} * g_j(x)\|_{L^2_x(V^r_k)} \leq C (D_1 + \|\hat{\phi}_{j,k}(\xi_j)\|_{V^r_k})\|g_j(x)\|_{L^2_x}
\end{equation}
from which we may conclude that
\[
B \leq C N^{\left(\frac{1}{2} - \frac{1}{r}\right)\frac{q}{q-2}} (D_1 + \sup_j\|\hat{\phi}_{j,k}(\xi_j)\|_{V^r_k})
\]
thus giving (\ref{fsbound}) after using the orthogonality of the $\tilde{f}_j.$

We now prove (\ref{caseN1}) which is the case $N=1$ of Proposition \ref{frequencyseparatedtheorem} and is similar to Lemma 3.4 of \cite{demeter09osm}.
Let $\psi$ be a Schwartz function with $\hat{\psi}$ supported on $[-1,1]$ and $\hat{\psi}(0) = 1,$ and write $\psi_k = 2^k\psi(2^k\cdot).$ It can be proven \cite{jones98oie},\cite{jones08svj} using L\'{e}pingles inequality for martingales that for every $g \in L^2$ (and here we use $r > 2$)
\begin{equation} \label{lepineq}
\|\psi_k * g(x)\|_{L^2_x(V^r_{k})} \leq C \|g\|_{L^2}\ \ .
\end{equation}
Then
\[
\|\tilde{\phi}_{j,k} * g(x)\|_{L^2_x(V^r_{k})} \leq \|\hat{\tilde{\phi}}_{j,k}(0)\psi_k * g(x)\|_{L^2_x(V^r_{k})} + \|(\tilde{\phi}_{j,k} - \hat{\tilde{\phi}}_{j,k}(0)\psi_k) * g(x)\|_{L^2_x(V^r_{k})}\ \ .
\]
Applying the inequality 
\[
\|a_kb_k\|_{V^r_k} \leq \|a_k\|_{V^r_k} \|b_k\|_{V^r_k}
\]
and (\ref{lepineq}) gives 
\[
\|\hat{\tilde{\phi}}_{j,k}(0)\psi_k * g(x)\|_{L^2_x(V^r_{k})} \leq C \|\hat{\tilde{\phi}}_{j,k}(0)\|_{V^r_k} \|g\|_{L^2} = C\|\hat{\phi}_{j,k}(\xi_j)\|_{V^r_k} \|g\|_{L^2}\ \ .
\]
Estimating $V^r$ by $l^2$ gives
\[
\|(\tilde{\phi}_{j,k} - \hat{\tilde{\phi}}_{j,k}(0)\psi_k) * g(x)\|_{L^2_x(V^r_{k})} \leq \|(\tilde{\phi}_{j,k} - \hat{\tilde{\phi}}_{j,k}(0)\psi_k) * g(x)\|_{L^2_x(l^2_{k})} \leq C \|g\|_{L^2}
\]
where the last inequality follows, in the usual way, by switching the order of integration, applying Plancherel's theorem, switching the order back, and using the fact that each $\tilde{\phi}_{j,k}- \hat{\tilde{\phi}}_{j,k}(0)\psi_k$ has mean zero and that $D_1 \leq 1.$
\end{proof}

\ \\
Finally, we will need the following variant of Proposition \ref{frequencyseparatedtheorem} involving multipliers of fixed scale.
For each $j = 1, \ldots, N$ and each $k \leq 0$, let $\varphi_{j,k}$ be a Schwartz function with $\hat{\varphi}_{j,k}$ supported on the interval $(\xi_j-\frac{1}{2},\xi_j + \frac{1}{2}).$ Let
\[
E_1 = \sup_{j,k,x} |\frac{d}{dx}\hat{\varphi}_{j,k}(x)|\ \ . 
\]
We then have

\begin{proposition} \label{nonvartheorem}
Suppose that for each $j,$ $\xi_{j+1} > \xi_j + 1$ and that $2 < r < q$. Then
\begin{equation} 
\|\sum_{j=1}^n \varphi_{j,k}*f\|_{L^2_x(V^q_{k \leq 0})} \leq C_{q,r} N^{(\frac{1}{2} - \frac{1}{r})\frac{q}{q-2}}(E_1 + \sup_{j=1, \ldots, N}\|\hat{\varphi}_{j,k}(\xi_j)\|_{V^r_{k \leq 0}}) \|f\|_{L^2}\ \ . 
\end{equation}
\end{proposition}

The proof is identical to that of Proposition \ref{frequencyseparatedtheorem}, except that one may use the $L^2$ bound for the Hardy-Littlewood maximal operator in place of L\'{e}pingle's inequality.

\subsection{General case} \label{rmsection}

Here, we will prove the following bound, which establishes Theorem \ref{vmt} at $p=2$.

\begin{theorem} \label{rmtheorem}
Suppose $2 < r < q$. Then, 
\[
\|\Delta_k[f](x)\|_{L^2_x(V^q_k)} \leq C(1 + \log(N)) N^{(\frac{1}{2} - \frac{1}{r})\frac{q}{q-2}}(D_{1} + \sup_{j = 1, \ldots, N}\| \sum_{|\omega| = 2^k}\hat{\phi}_{\omega}(\xi_j)\|_{V^r_k}) \|f\|_{L^2}.
\] 
\end{theorem}

\begin{proof} By inequality (\ref{mml2bound}) it suffices to prove the bound with the $\tilde{V}^q$ norm in place of the $V^q$ norm. Using monotone convergence, we may replace the $\tilde{V}^q$ norm by the $\tilde{V}^q_{k \in [a,b)}$ norm, where $[a,b)$ is an arbitrary finite interval of integers, and the constant is independent of $[a,b).$ For the remainder of the proof, we will usually supress $k \in [a,b)$ from the notation. 

For each $k$, let $R_k$ be the set of dyadic intervals of length $2^k$ which have nonempty intersection with $X = \{\xi_1, \ldots, \xi_N\}.$
Choosing $M$ so that $2^{M-1} < N \leq 2^M,$ we can find (after possibly enlarging $[a,b)$) a sequence $a = k_0 < \ldots < k_{2^M}=b$ so that $|R_k|$ is constant on each interval $[k_l,k_{l+1}).$

We write 
\[
f_{-1} = \sum_{\omega \in R_{k_0}} (1_{\omega}\hat{f})\check{\ }
\]
and
\[
f_l = \sum_{\omega \in R_{k_{l+1}}} (1_{\omega}\hat{f})\check{\ } - \sum_{\omega \in R_{k_{l}}} (1_{\omega}\hat{f})\check{\ }
\]
for  integers $l \in [0,2^M)$ so that the $f_l$ are orthogonal projections of $f$ and for each relevant $k$, $\Delta_k[f] = \Delta_k[\sum_{l = -1}^{l(k)} f_l]$ where $l(k)$ is the unique integer satisfying $k \in [k_{l(k)},k_{l(k)+1}).$

We then estimate
\begin{multline} \label{rmthreeterms}
\|\Delta_k[f](x)\|_{L^2_x(\tilde{V}^q_{k})} \leq \\
\|\Delta_k[f_{l(k)}](x)\|_{L^2_x(\tilde{V}^q_{k})} + \|\Delta_k[f_{-1}](x)\|_{L^2_x(\tilde{V}^q_{k})} + \|\Delta_k[\sum_{l \in [0,l(k))}f_l](x)\|_{L^2_x(\tilde{V}^q_{k})}.
\end{multline}

To bound $\|\Delta_k[f_{l(k)}](x)\|_{L^2_x(\tilde{V}^q_{k})},$
we first observe that for any function $g$ on $[a,b)$ we have
\[
\|g(k)\|_{\tilde{V}^q_{k \in [a,b)}} \leq C\left(\sum_{l = 0}^{2^M-1}\|g(k)\|_{V^q_{k \in [k_{l},k_{l+1})}}^2 \right)^{\frac{1}{2}}.
\]
Thus
\begin{equation} \label{orthogrm}
\|\Delta_k[f_{l(k)}](x)\|_{L^2_x(\tilde{V}^q_{k})} \leq 
C \left(\sum_{l = 0}^{2^M-1}\|\Delta_{k}[f_{l}](x)\|_{L^2_x(V^q_{k \in [k_{l},k_{l+1})})}^2 \right)^{\frac{1}{2}}.
\end{equation}
Breaking up each $R_{k_{l+1}-1}$ into two collections of intervals (and thus decomposing each $\Delta_{k}, k \in [k_{l},k_{l+1})$ into the sum of two operators) each of whose members are separated by distance $2^{k_{l+1}-1}$, one may use scale invariance to apply the bound from Proposition \ref{frequencyseparatedtheorem} for each $l$, obtaining
\[
\|\Delta_{k}[f_{l}](x)\|_{L^2_x(V^q_{k \in [k_{l},k_{l+1})})} \leq C_{q,r} N^{(\frac{1}{2} - \frac{1}{r})\frac{q}{q-2}}(D_1 + \sup_{j}\|\sum_{|\omega| = 2^k} \hat{\phi}_{\omega}(\xi_{j})\|_{V^r_{k}}) \|f_{l}\|_{L^2}
\]
One then uses orthogonality to see that the right hand side of (\ref{orthogrm}) is 
\[
\leq C N^{(\frac{1}{2} - \frac{1}{r})\frac{q}{q-2}}(D_1 + \sup_j\|\sum_{|\omega| = 2^k}\hat{\phi}_{\omega}(\xi_j)\|_{V^r_k})\|f\|_{L^2}.
\] 

We now consider the $f_{-1}$ term on the right hand side of (\ref{rmthreeterms}).
Here, we break up $R_{k_0}$ into four collections of intervals each of whose members are separated by distance $3\cdot 2^{k_{0}}$ and thus decompose $f_{-1}$ into the sum of four functions. 
Let $\tilde{f}_{-1}$ be one of these functions with associated intervals $\tilde{R}_{k_0}$ which we enumerate $\omega_{1}, \omega_2, \ldots$. Denote (say) the minimal element of $X \cap \omega_j$ as $\tilde{\xi}_j.$ We then set 
\[
\varphi_{j,k} = (\hat{\phi}_{\omega_{j,k}} \psi((\cdot - \tilde{\xi}_j)/2^{k_{0}}))\check{\ }
\]
where $\psi$ is a Schwartz function equal to $1$ on $[-1,1]$ and supported on $[-1.01,1.01]$, and where $\omega_{j,k}$ is the dyadic interval of length $2^k$ containing $\tilde{\xi_j}.$ This gives
\[
\Delta_{k}[\tilde{f}_{-1}] = \sum_{j} \varphi_{j,k}*\tilde{f}_{-1}
\]
for $k \in [a,b).$  
One may then use scale invariance to apply Proposition \ref{nonvartheorem}, thus obtaining 
\begin{align*}
\|\Delta_k[\tilde{f}_{-1}](x)\|_{L^2_x(\tilde{V}^q_{k})} &\leq C N^{(\frac{1}{2} - \frac{1}{r})\frac{q}{q-2}} (D_1 + \sup_j\|\sum_{|\omega| = 2^k}\hat{\phi}_{\omega}(\xi_j)\|_{V^r_k})\|\tilde{f}_{-1}\|_{L^2} \\
&\leq C N^{(\frac{1}{2} - \frac{1}{r})\frac{q}{q-2}} (D_1 + \sup_j\|\sum_{|\omega| = 2^k}\hat{\phi}_{\omega}(\xi_j)\|_{V^r_k})\|f\|_{L^2}.
\end{align*}

It remains to bound the last term on the right hand side of (\ref{rmthreeterms}). We need to show that
\begin{multline*}
\|\Delta_k[\sum_{l \in [0,l(k))}f_l](x)\|_{L^2_x(\tilde{V}^q_{k})} \\ \leq C (1 + \log(N)) N^{(\frac{1}{2} - \frac{1}{r})\frac{q}{q-2}} \sup_j(D_1 + \|\sum_{|\omega| = 2^k}\hat{\phi}_{\omega}(\xi_j)\|_{V^r_k})\|f\|_{L^2}.
\end{multline*}
For each $m = 0,\ldots,M-1$ and integers $l \in [0,2^M)$ let $\beta_{l,m} = \emptyset$ if $l$ is contained in the left child of the dyadic interval of length $2^{m+1}$ containing $l$, and otherwise (i.e. if $l$ is contained in the right child) let $\beta_{l,m}$ be the left child of the dyadic interval of length $2^{m+1}$ containing $l$. One may then check that
\[
[0,l) = \bigcup_{m=0}^{M-1} \beta_{l,m}
\] 
and that this union is disjoint.

We then have
\[
\|\Delta_k[\sum_{l \in [0,l(k))}f_l](x)\|_{L^2_x(\tilde{V}^q_k)} \leq \sum_{m = 0}^{M-1} \|\Delta_k[\sum_{l \in \beta_{l(k),m}}f_l](x)\|_{L^2_x(\tilde{V}^q_{k})}.
\]
For each $m$ in the sum above we have
\begin{multline} \label{sumoverdyadl}
\|\Delta_k[\sum_{l \in \beta_{l(k),m}}f_l](x)\|_{L^2_x(\tilde{V}^q_{k \in [a,b)})} \leq \\ C \left(\sum_{n = 0}^{2^{M-m} - 1} \|\Delta_k[\sum_{l \in \beta_{l(k),m}}f_l](x)\|^2_{L^2_x(V^q_{k \in [k_{n2^m},k_{(n+1)2^m}) })}\right)^{1/2}.
\end{multline}
If $n$ is even then $[n2^m,(n+1)2^m)$ is the left child of it's dyadic parent, and so $\beta_{l(k),m} = \emptyset$ for $k \in [k_{n2^m},k_{(n+1)2^m}).$ If $n$ is odd, then $\beta_{l(k),m} = [(n-1)2^m,n2^m)$ for $k \in [k_{n2^m},k_{(n+1)2^m}).$  One may then apply Proposition \ref{nonvartheorem} in the same manner as for the $f_{-1}$ term to see that 
\begin{multline*}
\|\Delta_k[\sum_{l \in \beta_{l(k),m}}f_l](x)\|_{L^2_x(V^q_{k \in [k_{n2^m},k_{(n+1)2^m}) })}\\ \leq C N^{(\frac{1}{2} - \frac{1}{r})\frac{q}{q-2}} (D_1 + \sup_j\|\sum_{|\omega| = 2^k}\hat{\phi}_{\omega}(\xi_j)\|_{V^r_k})\|\sum_{l \in [(n-1)2^m,n2^m)}f_l\|_{L^2}.
\end{multline*}
Thus, using orthogonality, the left side of (\ref{sumoverdyadl}) is
\[
\leq C N^{(\frac{1}{2} - \frac{1}{r})\frac{q}{q-2}} (D_1 + \sup_j\|\sum_{|\omega| = 2^k}\hat{\phi}_{\omega}(\xi_j)\|_{V^r_k})\|f\|_{L^2}\ \ .
\]
Summing over $m$ loses an additional factor of $M \leq 1 + \log(N)$, giving the desired bound. 
\end{proof}

\section{The $L^p$ bound} \label{czsection}

Letting
\[
\calV[f](x) = \|\sum_{|\omega| = 2^k} \phi_{\omega} * f(x)\|_{V^q_k} 
\]
and 
\[
A_M = (1 + \log(N)) N^{(\frac{1}{2} - \frac{1}{r})\frac{q}{q-2}}(D_{M} + \sup_{j = 1, \ldots, N}\| \sum_{|\omega| = 2^k}\hat{\phi}_{\omega}(\xi_j)\|_{V^r_k}), 
\]
we aim to prove the weak-type estimate below, from which Theorem \ref{vmt} will follow by interpolation with Theorem \ref{rmtheorem}
\begin{theorem} \label{RWT11} Suppose $2 < r < q.$ Then, there exists an $M$, depending only on $q,r,$ such that
for all $f\in L^1$ and $\lambda> 0$
$$|\{x: \calV[f](x) >\lambda\}|\leq C N^{1/2} A_{M} \|f\|_{L^1} \lambda^{-1}.$$
\end{theorem}

\begin{proof}[Proof of Theorem \ref{RWT11}]
Let $f\in L^1(\R)$. After renormalizing the $\phi_\omega,$ we may assume that $A_{M} = 1$, where $M$ will be determined later.  
Applying Theorem \ref{cztheorem} und using the notation there we set
$$E=\bigcup_{I\in \I} 5I$$
and obtain 
\begin{align*}
|\{x: \calV[f]>\lambda\}| &\le 
|E|  + |\{x: \calV[g](x) >\lambda /2\}|+ |\{x\in E^c: \calV[b](x)> \lambda/2\}|\\
&\le C N^{1/2} \|f\|_{L^1}/\lambda + C A_{M}^2\|g\|_{L^2}^2/\lambda^2 + C\|\calV[b]\|_{L^1(E^c)}/\lambda\\
&\le C N^{1/2} \|f\|_{L^1}/\lambda  + C \sum_I \|\calV[b_I]\|_{L^1((5I)^c)}/\lambda\ \ .
\end{align*}
Hence it remains to show that that for every $I\in \I$ we have
$$\|\calV[b_I]\|_{L^1((5I)^c))}\le C|I|\lambda\ \ . $$
By translation and dilation, assume without loss of generality that $I=[-1/2,1/2)$.
We have
\begin{align*}
\|\calV[b_I]\|_{L^1((5I)^c)} &=
\|\|\sum_{\substack{|\omega|=2^k \\ \omega\cap X\neq \emptyset}} 
 \phi_{\omega} * b_I(x)\|_{V^q_k}\|_{L^1_x((5I)^c)}\\ 
&\le C \sum_k \| \sum_{\substack{|\omega|=2^k \\ \omega\cap X\neq \emptyset}} 
 \phi_{\omega} * b_I \|_{L^1((5I)^c)}\ \ .
\end{align*}
We shall estimate each term separately in the sum over $k$. 

Let $\epsilon = \min(\frac{1}{2},\frac{1}{3}(\frac{1}{2} - \frac{1}{r})\frac{q}{q-2}).$ First consider $2^k<N^\epsilon$. We then estimate ($\sum_\omega$ will always denote the
sum over the collection of intervals ${|\omega|=2^k,\omega\cap X\neq \emptyset}$):

\begin{align} \label{nearfar}
& \|  \sum_\omega \phi_{\omega} * b_I \|_{L^1((5I)^c)} \nonumber \\ \le & \| \sum_\omega \phi_{\omega} * b_I(x)\|_{L^1(N^\epsilon 2^{-k(1+\epsilon)} I)} 
+ \| \sum_\omega \phi_{\omega} * b_I\|_{L^1((N^\epsilon 2^{-k(1+\epsilon)}I)^c)} \nonumber \\
\leq &
N^{\epsilon/2}2^{-k(1 + \epsilon)/2}
\|\sum_\omega \phi_{\omega} * b_I \|_{L^2} 
+ \| \sum_\omega \phi_{\omega} * b_I \|_{L^1((N^\epsilon 2^{-k(1+\epsilon)} I)^c)} \ \ .
\end{align}

We will first estimate the $L^2$ norm on the right-hand-side of (\ref{nearfar}).
For each $\omega$ with $X\cap\omega\neq \emptyset$ let $\xi_\omega$
denote the minimal element in $X\cap \omega$. Moreover denote
$$\tilde{\phi}_{\omega}(x)=\phi_{\omega}(x) e^{-2 \pi i\xi_\omega x}$$
Using the cancellation property, we have
\[
\int_{3I} 
b_I (y) {\phi}_{\omega}(x-y) \, dy 
=\int_{3I} 
b_I (y) [\tilde{\phi}_{\omega}(x-y)-\tilde{\phi}_{\omega}(x)] e^{2 \pi i\xi_\omega (x-y)}\, dy\ \ .
\]
Thus, writing 
\[
T_{\omega}[f](x) = \phi_{\omega}*f(x) - \phi_{\omega}(x) \int_{3I} e^{-2\pi i \xi_{\omega} y} f(y)\ dy
\]
we have
\[
\sum_{\omega} \phi_{\omega}*b_I = \sum_{\omega} T_{\omega}[f_I] - \sum_{\omega}T_{\omega}[g_I]\ \ .
\]
We will estimate the $L^2$ norms of two terms above separately.

Since the Fourier transforms of the $\phi_\omega$ are disjointly supported, we can estimate
\begin{equation} \label{orthdint}
\| \sum_{\omega} T_{\omega}[f_I] \|_{L^2} 
\leq 
(\sum_\omega  \| T_{\omega}[f_I] \|_{L^2}^2)^{1/2}.
\end{equation}
We have
\begin{align*}
|T_{\omega}[f_I](x)| &= |\int_{3I} f_I (y) [\tilde{\phi}_{\omega}(x-y)-\tilde{\phi}_{\omega}(x)] e^{2 \pi i\xi_\omega (x-y)}\, dy| \\
&\le \|f_I\|_{L^1} \sup_{y\in 3I} |\tilde{\phi}_{\omega}(x-y)-\tilde{\phi}_{\omega}(x)|\\
&\le C (N^{-1/2}|I|\lambda) 2^{2k} D_1 (1+\min(2^k,1) |x|)^{-1} \\
& \le C (N^{-1/2}|I|\lambda) 2^{2k} N^{-3\epsilon} (1+\min(2^k,1) |x|)^{-1}
\end{align*}
and hence 
\[ 
\|T_{\omega}[f_I](x)\|_{L^2} 
\le C (N^{-1/2}  |I|\lambda) 2^{3k/2} N^{-5 \epsilon/2}. 
\]
Above, we have used the normalization $A_M \leq 1$ to replace $D_1$ by $N^{-3\epsilon}.$
Finally, since there are at most $N$ terms in the sum over $\omega$, (\ref{orthdint}) gives
\[
\|\sum_{\omega} T_{\omega}[f_I]\|_{L^2}
\le C |I|\lambda 2^{3k/2} N^{-5\epsilon/2}.
\]
This estimate can be used for the $f_I$ part of the first term in (\ref{nearfar}) and
upon adding over $2^k<N^\epsilon$ results in the desired bound
for this part of the sum.

For the $g_I$ term, we have a worse bound on the $L^1$ norm, and thus cannot
use the same estimate. On the other hand $g_I$ is in $L^2$, so we can employ 
Hilbert space techniques.

Considering the $T_\omega$ as maps from $L^2(3I)$ to $L^2(\rea)$ and using the fact that the Fourier transforms of the $\phi_{\omega}$ are disjointly supported, we see that the ranges of $T_\omega$ are
pairwise orthogonal. Let $h$ be a function of norm $1$ such that $\|T^*h\|$ is within a 
factor of two of being maximal, and let $h_\omega$ be the orthogonal projection of $h$ onto the 
range of $T_\omega$ so that $\sum_\omega\|h_\omega\|^2\le \|h\|_2^2$. Then we have
\begin{align}\label{schur}
\|\sum_\omega T_\omega\|^2=\|\sum_\omega T^*_\omega\|^2
\le & 4\sum_{\omega,\omega'} \<h, T_\omega, T^*_{\omega'}h\>\nonumber \\ 
= & 4\sum_{\omega,\omega'} \<h_\omega , T_\omega, T^*_{\omega'} h_{\omega'}\>\nonumber\\
\le & 4\sum_{\omega,\omega'} \|h_\omega\| \| T_\omega, T^*_{\omega'}\| \|h_{\omega'}\|\nonumber \\
\le &4\|(\|T_\omega T_{\omega'}\|)_{\omega,\omega'}\| \sum_\omega \|h_\omega\|^2\nonumber \\
\le & 4\sup_{\omega} \sum_{\omega'} \| T_\omega T^*_{\omega'}\| \ \ .
\end{align}
In the last line we have used Schur's test on the norm of the matrix 
$(\|T_\omega T_{\omega'}\|)_{\omega,\omega'}$ acting on the space $l^2(\{\omega: |\omega|=2^k,\omega\cap X\neq 0\})$, which is a consequence of interpolation
between the trivial $l^1$ and $l^\infty$ bounds.

 Reusing the $L^1(3I) \rightarrow L^2(\rea)$ bound for $T_{\omega}$ employed to estimate the $f_I$, we bound the diagonal terms
\[
\|T_{\omega}T^*_{\omega}\| \leq |I| 2^{3k} N^{-5\epsilon} \leq 2^{2k}N^{-4\epsilon}. 
\]
For the off-diagonal terms we calculate 
\[
T_\omega T^*_{\omega'}[f](x) = \int_{\rea} K(x,z) f(z)\ dz
\]
where
\[
K(x,z) = \int_{3I} [\tilde{\phi}_\omega(x-y)-\tilde{\phi}_\omega(x)]e^{2 \pi i\xi_\omega (x-y)}
[\overline{\tilde{\phi}}_{\omega'}(z-y)-\overline{\tilde{\phi}}_{\omega'}(z)] e^{-2 \pi i\xi_{\omega'} (z-y)}\, dy.
\]
The absolute value of the display above is 
$$ |\int_{3I} [\tilde{\phi}_\omega(x-y)-\tilde{\phi}_\omega(x)]
[\overline{\tilde{\phi}}_{\omega'}(z-y)-\overline{\tilde{\phi}}_{\omega'}(z)] e^{2 \pi i(\xi_{\omega'}-\xi_\omega) y}\, dy|. $$
After one partial integration, we see that this is bounded above by
\begin{multline*}
\frac 1{|\xi_\omega -\xi_{\omega'}|}
\int_{3I} |([\tilde{\phi}_\omega(x-y)-\tilde{\phi}_\omega(x)]
[\tilde{\phi}_{\omega'}(z-y)-\tilde{\phi}_{\omega'}(z)])'|dy \\
+\frac 1{|\xi_\omega -\xi_{\omega'}|}
|[\tilde{\phi}_\omega(x-r)-\tilde{\phi}_\omega(x)]
[\tilde{\phi}_{\omega'}(z-r)-\tilde{\phi}_{\omega'}(z)]| \\ 
+\frac 1{|\xi_\omega -\xi_{\omega'}|}
|[\tilde{\phi}_\omega(x-l)-\tilde{\phi}_\omega(x)]
[\tilde{\phi}_{\omega'}(z-l)-\tilde{\phi}_{\omega'}(z)]| 
\end{multline*}
when $\omega\neq \omega',$ where $r$ and $l$ are the right and left endpoint of $3I$.
This however is 
$$\leq C \frac 1{|\xi_\omega -\xi_{\omega'}|} 2^{4k}D_1^2(1+\min(2^k,1)|x|)^{-1} (1+\min(2^k,1)|z|)^{-1}\ \ . $$
Estimating the operator norm by the Hilbert-Schmidt norm shows
$$\|T_{\omega'}^*T_\omega\|\le 
C \frac 1{|\xi_\omega -\xi_{\omega'}|} N^{-5\epsilon}2^{3k}\ \ .$$

Since every dyadic interval of length $1$ contains at most one of the frequencies $\xi_\omega$,
we obtain from (\ref{schur})
\[
\|\sum_{\omega}T_{\omega}\| \leq C (1 + \log(N)) N^{-2\epsilon} 2^k.
\]  
This estimate can be used for the $g_I$ part of the
first term in (\ref{nearfar}) and upon adding over $2^k<N^\epsilon$ results in the desired bound for this part of the sum.

The second term in (\ref{nearfar}) is a bit easier: we have
\begin{align*}
& \| \sum_\omega \phi_{\omega} * b_I \|_{L^1((N^\epsilon 2^{-k(1+\epsilon)} I)^c)} \\
\leq &\  \sum_{\omega}\|\int_{3I}b_I(y) \phi_{\omega}(x - y)\ dy\|_{L^1_x((N^\epsilon 2^{-k(1+\epsilon)} I)^c)} \\
\leq &\  \sum_{\omega} \|b_I\|_{L^1} \|\sup_{y \in 3_I} |\phi_{\omega}(x - y)|\|_{L^1_x((N^\epsilon 2^{-k(1+\epsilon)} I)^c)} \\
\leq &\ C N \|b_I\|_{L^1} D_M\|(1 + 2^k \dis(x,3I))^{-M}\|_{L^1_x((N^\epsilon 2^{-k(1+\epsilon)} I)^c)}  \\
\leq &\ C N^{1 - (M+2)\epsilon}2^{k\epsilon(M-1)} \|b_I\|_{L^1}.
\end{align*}
Choosing $M$ satisfying $1/\epsilon \leq M < 1 + 1/\epsilon$, the sum over $2^k \leq N^{\epsilon}$ is
$$\le  C \lambda |I|,$$
which is the desired estimate.

Using the same method as in the previous paragraph, we estimate the terms $2^k\ge N^{\epsilon}$. This gives
\begin{align*}
\|\sum_\omega \phi_{\omega} * b_I \|_{L^1((5I)^c)} &\leq 
C N \|b_I\|_{L^1} D_M\|(1 + 2^k \dis(x,3I))^{-M}\|_{L^1_x((5I)^c)}  \\
&\leq C N^{1 - 3\epsilon} 2^{-k(M-1)} \|b_I\|_{L^1}.
\end{align*}
the sum over $2^k > N^{\epsilon}$ is again
$$\le  C \lambda |I|. $$
\end{proof}

\providecommand{\bysame}{\leavevmode\hbox to3em{\hrulefill}\thinspace}
\providecommand{\href}[2]{#2}

\end{document}